\documentclass{article}

\usepackage[utf8]{inputenc}
\usepackage[margin=1in]{geometry} 
\usepackage{amsmath,amsthm,amssymb,amsfonts, enumitem, fancyhdr, color, verbatim, graphicx, environ,bbm,subcaption}
\usepackage{mathtools}
\usepackage {tikz}
\usetikzlibrary {positioning}
\usepackage{ulem}
\usepackage{parskip}
\usepackage{soul}
\usepackage{setspace}
\usepackage[color=mnt,textsize=footnotesize]{todonotes}

\newtheorem{theorem}{Theorem}
\newtheorem{conj}[theorem]{Conjecture}
\newtheorem{cor}[theorem]{Corollary}
\newtheorem{fact}[theorem]{Fact}
\newtheorem{lemma}[theorem]{Lemma}
\newtheorem{prob}[theorem]{Problem}
\newtheorem{qst}[theorem]{Question}

\title{Erd\H{o}s-Ko-Rado Theorems for Paths in Graphs}
\author{
Neal Bushaw \thanks{
Department of Mathematics and Applied Mathematics,
Virginia Commonwealth University, USA
}
\and
James Danielsson \footnotemark[1]
\and
Glenn Hurlbert \footnotemark[1]
}
\date{}

\def\a{{\alpha}}

\def\D{{\Delta}}
\def\m{{\mu}}
\def\t{{\tau}}
\def\w{{\omega}}

\def\cA{{\cal A}}
\def\cB{{\cal B}}
\def\cC{{\cal C}}
\def\cF{{\cal F}}

\def\cH{{\cal H}}
\def\cI{{\cal I}}
\def\cP{{\cal P}}
\def\cR{{\cal R}}
\def\cS{{\cal S}}
\def\cM{{\cal M}}
\def\cK{{\cal K}}

\def\zF{{\mathbb F}}
\def\zP{{\mathbb P}}

\def\zpq{{\zP(q)}}

\def\zprm{{\zP(r-1)}}
\def\zZ{{\mathbb Z}}
\def\mt{{\emptyset}}

\def\deg{{\sf deg}}
\def\ekr{{\sf EKR}}
\def\gir{{\sf gir}}
\def\im{{\sf im}}

\def\orP{{\overrightarrow{P}}}
\def\olP{{\overleftarrow{P}}}

\def\snt{{S_n^t}}

\def\dist{{\sf dist}}

\DeclarePairedDelimiter\ceil{\lceil}{\rceil}
\DeclarePairedDelimiter\floor{\lfloor}{\rfloor}

\definecolor{brwn}{RGB}{140, 70, 20}
\definecolor{gren}{RGB}{  0,140, 10}
\definecolor{mnt}{RGB}{167,230,215}
\definecolor{orn}{RGB}{253,210,125}
\definecolor{sky}{RGB}{220,245,255}
\definecolor{commiepinko}{RGB}{254,1,154}

\newcommand{\nb}[1]{\textcolor{commiepinko}{\sf{#1}}}
\newcommand{\jd}[1]{\textcolor{gren}{\sf{#1}}}
\newcommand{\gh}[1]{\textcolor{blue}{\sf{#1}}}
\newcommand{\up}[1]{\textcolor{red}{\sf{#1}}}

\newcommand{\fract}[2]{\frac{}{}}

\begin{document}

\maketitle

\begin{abstract}
A family of sets is $s$-intersecting if every pair of its sets has at least $s$ elements in common.
It is an $s$-star if all its members have some $s$ elements in common.
A family of sets is called $s$-\ekr\ if all its $s$-intersecting subfamilies have size at most that of some $s$-star.
For example, the classic 1961 Erd\H{o}s-Ko-Rado theorem states essentially that the family of $r$-sized subsets of $\{1,2,\ldots,n\}$ is $s$-\ekr\ when $n$ is a large enough function of $r$ and $s$, and the 1967 Hilton-Milner theorem provides the near-star structure of the largest non-star intersecting family of such sets.
Two important conjectures along these lines followed: by Chv\'atal in 1974, that every subset-closed family of sets is 1-\ekr, and by Holroyd and Talbot in 2005, that, for every graph, the family of all its $r$-sized independent sets is 1-\ekr\ when every maximal independent set has size at least $2r$.

In this paper we present similar 1-\ekr\ results for families of length-$r$ paths in graphs, specifically for sun graphs, which are cycles with pendant edges attached in a uniform way, and theta graphs, which are collections of pairwise internally disjoint paths sharing the same two endpoints.
We also prove $s$-\ekr\ results for such paths in suns, and give a Hilton-Milner type result for them as well.
A set is a transversal of a family of sets if it intersects each member of the family, and the transversal number of the family is the size of its smallest transversal.
For example, stars have transversal number 1, and the Hilton-Milner family has transversal number 2. We conclude the paper with some transversal results involving what we call triangular families, including a few results for projective planes.
\end{abstract}

\newpage

\section{Introduction}
\label{s:Intro}

Let $\cF$ be a family of sets, and use the shorthand notations $\cup\cF=\cup_{A\in\cF}A$ and $\cap\cF=\cap_{A\in\cF}A$.
We say that $\cF$ is $s$-{\it intersecting} if $|A\cap B|\ge s$ for all $A,B\in\cF$, and $\cF$ is an $s$-{\it star} if $|\cap\cF|\ge s$.
By {\it intersecting} we mean 1-intersecting, and by {\it star} we mean 1-star.
We say that $\cF$ is {\it exactly} $s$-intersecting if $|A\cap B|=s$ for all $A,B\in\cF$.
For $X\subseteq\cup\cF$ with $|X|=s$, we define the {\it full $s$-star of $\cF$ at $X$} to be $\cF_X=\left\{A\ \mid\ X\subseteq A\in\cF\right\}$.
We also write $\cF^r=\left\{A\in\cF\ \mid\ \left|A\right|=r\right\}$. 
For each of these, note that the definition is trivial in the case $s=0$, and so throughout we will implicitly assume that $s>0$.
We say that $\cF$ is $s$-\ekr\ if some $s$-star $\cF_x$ satisfies $|\cF_x|\ge |\cH|$ for all $s$-intersecting $\cH\subseteq\cF$, and {\it strictly $s$-\ekr} if every largest $s$-intersecting family is an $s$-star, dropping the mention of $s$ when it equals 1.
We extend this language to graphs as follows.
Let $G$ be a graph and $\cF$ be a family of vertex sets of subgraphs of $G$.
We say that $G$ is $(\cF,s)$-\ekr if $\cF$ is $s$-\ekr\ ($\cF$-\ekr\ if $s=1$).  

One of the cornerstones of extremal set theory is the following result of Erd\H{o}s, Ko, and Rado.

\begin{theorem}
\label{t:EKR}
\cite{ErdKoRad}
There is a function $n_0(r,s)$ such that, if $\cF \subseteq \binom{[n]}{r}$ is $s$-intersecting and $n\ge n_0(r,s)$, then $|\cF| \le \binom{n-s}{r-s}$.
Moreover, equality holds if and only if $\cF$ is an $s$-star.
\end{theorem}

For $s=1$ we have $n_0(r,s)=2r+1$, with the statement for $n=2r$ being true without the uniqueness of the star extremum.
For $s>1$ and $n<n_0(r,s)$, the sizes and structures of maximum $s$-intersecting families were determined in the Complete Intersection Theorem of \cite{AhlsKhac}.

It is natural to ask which families of sets $\cF$ give a theorem similar to Theorem \ref{t:EKR}; that is, which $\cF$ are \ekr? 
Chv\'atal \cite{Chvat} made the following conjecture.

\begin{conj}
\label{c:Chvatal}
\cite{Chvat}
Every subset-closed family of sets is \ekr.
\end{conj}

One of the more well-studied families along these lines is the family of independent sets of a graph $G$, which we denote by $\cI(G)$.
This study was formally initiated by Holroyd, Spencer and Talbot in \cite{HolSpeTal}, although earlier traces of the idea are found in \cite{Berge,BollLead,DezaFran}.
In \cite{HolrTalb} we find the following conjecture.
Let $\m(G)$ denote the {\it independent domination number} of $G$; i.e., the minimum size of a maximal independent set of $G$.

\begin{conj}
\label{c:HST}
\cite{HolrTalb}
Let $G$ be any graph and suppose that $1\le r\le \mu(G)/2$.
Then $G$ is $\cI^r$-\ekr, and is strictly so if $2 < r < \mu(G)/2$.
\end{conj}

There is a growing literature affirming this conjecture --- see \cite{BorgFegh,FegHurKam,FranHurl,Hurlbert} for recent examples and a survey containing history and past references.

Simonovits and S\'os \cite{SimoSos1,SimoSos2} introduced the problem of finding the size of the largest family of subgraphs of the complete graph $K_n$ such that the intersection of every pair is in a given set $L$ of graphs.\footnote{We note the subtle difference here between a family of paths, and a family of vertex sets of paths. In this paper, we consider families of paths, though for trees, cycles, suns, and theta graphs this distinction makes no difference since any $r$-set of vertices contains at most one path.}
They proved results on several classes for $L$, such as $k$-vertex graphs, complete graphs, stars, connected graphs, paths, and cycles.
Ellis, Filmus, and Friedgut \cite{EllFilFri} solved the problem when $L$ is the family of triangles ($K_3$).
Recently, Frankl, et al. \cite{FHIKLMP} proved results when the family of subgraphs of $K_n$ is restricted to spanning trees and $L$ is the set of all edges.
In this paper, we expand our attention to subgraphs of any graph $G$ of interest, and then restrict our attention to the case that $\cF\in\{\cP(G),\cP^r(G)\}$; i.e., $\cF$ is either the family of all paths in a graph $G$, or the family of all paths on $r$ vertices in $G$, with $L$ being the set of all vertices of $G$.
Observe that, unlike $\cI(G)$, $\cP(G)$ is not subset-closed, as not all subgraphs of paths are paths.
This difference is critical because it prohibits the use of algebraic techniques that apply to simplicial complexes, for example.
In particular, this problems sits outside of the realm for Chv\'atal's Conjecture.
Indeed, Theorem \ref{t:SunEKRnon} provides an instance of a non-\ekr\ family in our scheme.

In fact, we will tacitly assume that $G$ is connected throughout, since an intersecting subfamily must always live inside a component of $G$.
For $x\in\cup\cF$ define $\deg(x)=|\cF_x|$, and denote $\D(\cF)=\max_{x\in\cup\cF}\deg(x)$.
More generally, we write $\D_s(\cF)=\max_{|X|=s}|\cF_X|$.
We say that $\cF$ is {\it $r$-uniform} if $|A|=r$ for all $A\in\cF$ (i.e., $\cF=\cF^r$), and write $\cF-x=\{A-\{x\}\ \mid\ A\in\cF\}$. 
A {\it transversal} of a family $\cF$ is a set $X$ such that $X\cap F\neq\emptyset$ for each $F\in\cF$; we denote by $\t(\cF)$ the minimum size of a transversal of $\cF$, called the {\it transversal number} of $\cF$.
The set $\cup\cF$ is the {\it trivial} transversal of $\cF$; all others are {\it nontrivial}.

Theorem \ref{t:EKR} can be expressed using this transversal language; it says that, when $n$ is large enough in terms of $r$, among the maximum size intersecting families is a family with transversal number one.  In general, an \ekr-type theorem gives a condition under which there is guaranteed to be a maximum size intersecting family with transversal number one.

Hilton and Milner considered the problem of determining the size of the largest non-star intersecting family under the same conditions as the \ekr\ Theorem. 
Let $X=\{2,\ldots,r+1\}$ and define $\cM=\{H\in\binom{[n]}{r}_1 \mid H\cap X\not=\mt\} \cup \{X\}$.
Furthermore, define $\cK=\{H\in\binom{[n]}{r} \mid |H\cap [3]|\ge 2\}$.
In \cite{HiltMiln} they proved the following theorem.

\begin{theorem}[\cite{HiltMiln}]
\label{t:HM}
If $n\ge 2r$, $\cF\subseteq \binom{[n]}{r}$ is intersecting, and $\cF$ is not a star then $|\cF|\le |\cM| = \binom{n-1}{r-1}-\binom{n-r-1}{r-1}+1$.
Additionally, equality holds if and only if $\cF\cong\cM$ or $r=3$ and $\cF\cong\cK$.
\end{theorem}

Since $\cM$ is not a star, the Hilton-Milner Theorem states that, for large enough $n$, among the intersecting families with transversal number at least two, $\cM$ has maximum size.

In Section \ref{s:HardPaths} we prove some \ekr\ results for paths in graphs, beginning with a few elementary results, followed by our main \ekr\ theorems for paths in {\it sun} graphs and {\it theta} graphs.
A sun graph is formed from a cycle by adding a fixed number of pendants to each vertex, and theta graphs are unions of internally vertex disjoint paths between two fixed vertices  (see their formal definitions in Subsections \ref{ss:Suns} and \ref{ss:Theta}).
We prove in Theorems \ref{t:Sun-S-IntersectingEKRunif} and \ref{t:ThetaEKR}, respectively, that each of these classes are \ekr, for the appropriate range of path lengths, and also that suns are $s$-\ekr\ similarly.
Then we present in Section \ref{s:Trans} some transversal results involving what we call triangular families, including a few results for projective planes.
We finish with several interesting questions in Section \ref{s:Quests}.


\section{Path \ekr\ Results}
\label{s:HardPaths}


\subsection{Elementary Results}
\label{ss:EasyPaths}

For a graph $G$, recall that $\cP=\cP(G)$ is the set of all paths in $G$, and $\cP^r=\cP^r(G)$ is the set of all paths in $G$ with exactly $r$ vertices. 
The following theorem was proved by Berge \cite{BergeGraphs}.

\begin{theorem}[\cite{BergeGraphs}]
\label{t:trees}
Let $T$ be a tree.
Then every intersecting family of paths in $T$ is a star.
\hfill $\Box$
\end{theorem}

\begin{theorem}
\label{t:girth}
Let $G$ be a graph with girth $\gir(G)>3l$.
Then every intersecting family of paths of length at most $l$ in $G$ is a star.
\end{theorem}

\begin{proof}
If the union of some set of $k$ paths of length at most $l$ contains a cycle of length at least $3l$, then $k\ge 4$, and therefore is not intersecting.
Hence the union of an intersecting family $\cF$ of paths of length at most $l$ is a tree (since it is connected).
Thus $\cF$ is a star by Theorem \ref{t:trees}.
\end{proof}

\begin{cor}
\label{c:TreeGirthEKR}
If $T$ is a tree then $T$ is $\cP$-\ekr.
If $\gir(G)>3r-3$ then $G$ is $\cP^r$-\ekr.
\end{cor}

Observe that the complete graph $K_3$ is not $\cP^2$-\ekr.
Moreover, any graph containing a triangle and having maximum degree 3 is $\cP^2$-\ekr, but not strictly so.
Among maximum degree 3 graphs, the converse is also true.

\begin{fact}
\label{f:edge}
If $\D(G)=3$ and $G$ is triangle-free then $G$ is strictly $\cP^2$-\ekr.
\end{fact}

\begin{proof}
We prove the contrapositive.
Suppose that $G$ is not strictly $\cP^2$-\ekr; that is, some maximum size intersecting family is not a star.
Because every intersecting family of edges of $G$ is either a triangle or a star, this implies that in this family some triangle is at least as large as every star --- i.e., $\D(G)\le 3$.
\end{proof}

Because of Fact \ref{f:edge} we will assume throughout the remainder of this manuscript that $r\ge 3$.


\subsection{Suns}
\label{ss:Suns}

\begin{figure}[]
    \centering
    \begin{subfigure}[b]{0.48\textwidth}
        \centering
\begin{tikzpicture}[scale=0.8,
    vertex/.style={circle, draw, minimum size=6mm, inner sep=0pt, font=\small}]
    \def \r {2.5} 
    \def \margin {7}
    \foreach \i [evaluate=\i as \ang using 90-\i*72] in {0,...,4} {
        \node[vertex] (v\i) at (\ang:\r) {$v_{\i}^0$};
        \draw ({72 * (\i - 1)+18+\margin}:\r)
            arc ({72 * (\i - 1)+18+\margin}:{72 * (\i)+18-\margin}:\r);
        \foreach \j [evaluate=\j as \rayang using \ang+40-20*\j] in {1,2,3} {
            \node[vertex] (v\i-\j) at (\rayang:4.2) {$v_{\i}^{\j}$};
            \draw (v\i) -- (v\i-\j);
        }
    }
\end{tikzpicture}
        \caption{The sun graph $S_5^3$.}
        \label{fig:sun_graph}
    \end{subfigure}
    \hfill 
    \begin{subfigure}[b]{0.48\textwidth}
        \centering
\begin{tikzpicture}[scale=0.8,
    vertex/.style={circle, draw, minimum size=6mm, inner sep=0pt, font=\footnotesize},
    strand label/.style={font=\footnotesize}]
    \node[vertex] (u) at (-4, 0) {$u$};
    \node[vertex] (v) at (4, 0) {$v$};
    \node[vertex] (w1-1) at (0, 2.5) {$w_{1,1}$};
    \draw (u) to [out=55, in=190] (w1-1);
    \draw (w1-1) to [out=-10, in=125] (v);
    \node[strand label] at (-2.5, 1.25) {$P_{(1)}$};
    \node[vertex] (w2-1) at (0, 0) {$w_{2,1}$};
    \draw (u) -- (w2-1) -- (v);
    \node[strand label] at (-2.5, -0.25) {$ P_{(2)}$};
    \node[vertex] (w3-1) at (-1.5, -2.5) {$w_{3,1}$};
    \node[vertex] (w3-2) at (1.5, -2.5) {$w_{3,2}$};
    \draw (u) to [out=-65, in=170] (w3-1);
    \draw (w3-1) to [out=-5, in=185] (w3-2);
    \draw (w3-2) to [out=10, in=-115] (v);
    \node[strand label] at (-2.5, -1.75) {$P_{(3)}$};
\end{tikzpicture}
        \caption{The generalized theta graph $\Theta(2,2,3)$.}
        \label{fig:theta_graph}
    \end{subfigure}
    \caption{Two structural graph representations.}
    \label{fig:SunTheta}
\end{figure}

Define the {\it sun} with $t$ rays $\snt$ to have vertex set $\{v_i^j \mid 0\le i<n, 0\le j\le t\}$, with edges $v_i^0v_{i+1}^0$, and $v_i^jv_i^0$ for all $i\in [n]$, and $0<j\leq t$.
We imagine the vertices $v_0^0, v_1^0, \ldots, v_{n-1}^0$ to be written clockwise around a circle --- see Figure \ref{fig:sun_graph}.

\begin{theorem}
\label{t:Sun-S-IntersectingEKRunif}
Let $3\le s+2\le r\leq \floor{\frac{n+s-1}{2}}$ and $t\ge 0$.
Then $\snt$ is $(\cP^r,s)$-\ekr, and strictly so when $r<(n+s-1)/2$ and it is not the case that $s=2$ and $t=1$.
Define the function
\[
f(r,s,t)= (r-s+1)+2t(r-s)+
\begin{cases}
\binom{t}{2}(r-s-1) = 3+4t+\binom{t}{2}&\rm{when\ }r=3,s=1,\\
t^2(r-s) = 4r-8 &\rm{when\ } r\ge 4, t=1, s=2, \rm{and}\\
t^2(r-s-1) &\rm{otherwise}.
\end{cases}
\]
If $\cF\subseteq\cP^r(\snt)$ is an $s$-intersecting family then $|\cF| \le f(r,s,t)$.
\end{theorem}

\begin{proof}
When $t=0$, the theorem is equivalent to Katona's Cycle Lemma \cite{Katona}, so we will assume throughout that $t>0$. 
We begin by displaying an $s$-intersecting family of maximum size, after which we will prove the upper bound.
Finally we will show that only a $s$-star can achieve that bound, with the exceptions of when $r=(n+s-1)/2$ or when $s=2$ and $t=1$.

Here we define an example of an $s$-intersecting family of maximum size.  
Let $C$ be the cycle in $S_n^t$, let $X=\{v_1^0,\ldots,v_{r-s}^0\}$, and consider the star $\cF^*$ formed by all $r$-paths containing $X$. 
There are $(r-s+1)$ such paths with both endpoints on $C$, $2t(r-s)$ paths with exactly one endpoint on $C$, and either $t^2(r-s-1)$ (when $r\ge 4$) or $\binom{t}{2}(r-s-1)$ (when $r=s+2=3$) paths with neither endpoint on $C$. 
Then $|\cF^*|=(r-s+1)+2t(r-s)+t^2(r-s-1)$ when $r\ge 4$, or $(r-s+1)+2t(r-s)+\binom{t}{2}(r-s-1)$ when $r=s+2=3$.

Now, we introduce some notation.
Let $\orP_i^{(j,k)}$ be the clockwise path of length $r-1$, beginning at $v_i^j$ and ending at $v_{i'}^k$, for the appropriate $i'$, and $\olP_i^{(j,k)}$ be the counter-clockwise path of length $r-1$, beginning at $v_{i+s-2}^k$ and ending at $v_{i'}^j$, for the appropriate $i'$.
Observe that for $r=3$ the definitions require that $i'=i$ and $j\not=k$ whenever both $j>0$ and $k>0$.
Note that this yields four types of paths, depending on whether or not each of $j$ and $k$ is zero.
More precisely, define the following four families that partition the set of all paths of length $r-1$ in $\snt$: 
\begin{itemize}
    \item $\cP^{(0,0)}=\{\orP_i^{(0,0)}\mid i\in [n]\}$,
    \item $\cP^{(0,1)}=\{\orP_i^{(0,k)}\mid i\in [n], 0<k\leq t\}$,
    \item $\cP^{(1,0)}=\{\orP_i^{(j,0)}\mid i\in [n], 0<j\leq t\}$, and
    \item $\cP^{(1,1)}=\{\orP_i^{(j,k)}\mid i\in [n], 0<j,k\leq t\}$.
\end{itemize}
Notice that, for each $i$, $j$, and $k$, $\orP_i^{(j,k)}=\olP_{i'}^{(k,j)}$ for some $i'$, and that, since $r\le \lfloor\frac{n+s-1}{2}\rfloor$,  $|\orP_i^{(j,k)}\cap\olP_i^{(j',k')}|<s$ for each $i,j,j',k$ and $k'$.  Using the indicator function $\mathbbm{1}_{x}$ ($=1$ if $x>0$ and $=0$ otherwise),
we define $I(j,k)=(\mathbbm{1}_j,\mathbbm{1}_k)$.
Then we see that a path $\orP_i^{(j,k)}\in \cP^{I(j,k)}$. 

Now we prove the upper bound.
Let $\cF$ be an $s$-intersecting family, and define $\cF^{(a,b)}=\cF\cap\cP^{(a,b)}$, for each $a,b\in\{0,1\}$.
Then $|\cF|=|\cF^{(0,0)}|+|\cF^{(0,1)}|+|\cF^{(1,0)}|+|\cF^{(1,1)}|$.
We calculate an upper bound on $|\cF|$ by providing upper bounds on each of these four subfamilies.
That is, given $a,b\in\{0,1\}$, we count the maximum number of paths in $\cF^{(a,b)}$.

We calculate $|\cF^{(a,b)}|$ as follows.
Select some $\orP_x^{(j,k)}\in\cF^{(a,b)}$ and consider the maximum number of paths that intersect $\orP_x^{(j,k)}$.
Any such path can be written as either $\orP_y^{(j',k')}$ or $\olP_y^{(j',k')}$ for some $v_y^0\in\{v_{x}^0,\ldots, v_{x+r-s}^0\}$ and $0\le j,j',k,k'\le t$.
Note for each such $v_y^0$ that, since $\cF^{(a,b)}$ is $s$-intersecting, $\cF^{(a,b)}$ has no two paths of opposite orientation $\orP_y^{(j',k')}$ or $\olP_y^{(j',k')}$; thus $\cF^{(a,b)}$ contains paths of only one of the orientations $\orP_y^{(j',k')}$ or $\olP_y^{(j',k')}$.

We count the number of clockwise oriented paths that intersect $\orP_x^{(j,k)}$ by summing the number of such paths $\orP_y^{(j',k')}$ over all $y$ such that $v_y^0\in\{v_{x}^0,\ldots, v_{x+r-s+1}^0\}$.
As we shall see below, the vertex $v_{x+r-s+1}^0$ will only matter in the case $s=2$, $t=1$, $(a,b)=(1,1)$.
We define $t(a,b)=\binom{t}{2}$ when $r=s+2=3$, and $t(a,b)=t^{a+b}$ otherwise.
Then set
\[
f^{(a,b)}(r,s,t) = (r-a-b-s+1)t(a,b).
\]

When $s=2$, $t=1$, and $(a,b)=(1,1)$, the number of non-pendant vertices $v_y^0\in\{v_{x}^0,\ldots, v_{x+r-1}^0\}$ equals $r-2$ --- the paths $\orP_x^{(1,1)}$ and $\orP_{x+r-1}^{(1,1)}$ 2-intersect because of the shared pendant vertex.
(The same occurs when $s=2$, $t=2$, and $(a,b)=(1,1)$, but the loss of additional choices for pendant vertices cancels any improvement.)
For any other choices of $s,t,a,b$, the number of non-pendant vertices $v_y^0\in\{v_{x}^0,\ldots, v_{x+r-s}^0\}$ equals $r-a-b-s+1$, and the number of choices for pendant vertices at each end of a path in $\orP_y^{(j',k')}$ equals $t(a,b)$.

If $s=2$, $t=1$, and $(a,b)=(1,1)$, we therefore have $|\cF^{(1,1)}|\le f^{(1,1)}(r,2,1)+1=r-2$ if $\cF$ is not a 2-star and $|\cF^{(1,1)}|\le f^{(1,1)}(r,2,1)=r-3$ if $\cF$ is a 2-star.
Otherwise, we have $|\cF^{(a,b)}|\le f^{(a,b)}(r,s,t)$,
and hence 
\begin{align}
|\cF|
&= |\cF^{(0,0)}|+|\cF^{(0,1)}|+|\cF^{(1,0)}|+|\cF^{(1,1)}|
\label{e:Fpart2}\\
&\leq (r-s+1)+2t(r-s)+t^2(r-s-1)
\nonumber
\end{align}
when $r\ge s+3$ or $r\ge 4$, with $|\cF|\le (r-s+1)+2t(r-s)+\binom{t}{2}(r-s-1) = 3+4t+\binom{t}{2}=f(r,s,t)$ when $r=s+2=3$, as noted above.
Hence $\snt$ is $\cP^r$-\ekr.

Next we show that if $r< (n+s-1)/2$ then only a star achieves the upper bound.
We will do this by proving that, with the exception of the case $s=2$, $t=1$, and $(a,b)=(1,1)$, each of the families $\cF^{(a,b)}$ is a star.
In addition, each of these families will have the same star center, 
which will force a star even in the cases when $s=1$ and $r=n/2$ or when $s=2$, $t=1$, and $(a,b)=(1,1)$.

Let $r< (n+s-1)/2$ and consider $\cF^{(a,b)}$ with $a,b\in\{0,1\}$ --- in the case that $r=(n+s-1)/2$ and $s=1$ then suppose that $(a,b)\not=(0,0)$.

First, suppose that it is the case that $s=2$, $t=1$, and $(a,b)=(1,1)$.
Assume that $|\cF|=f(r,2,1)$; then $|\cF^{(1,1)}|=f^{(1,1)}(r,2,1)=r-2$.
Suppose that $F\in\cF^{(1,1)}$.
Then $F=\orP_x^{(1,1)}$ for some $x$.
Because $r<\left\lfloor (n+1)/2\right\rfloor$, we have $\cF^{(1,1)}\not=\cP^{(1,1)}$, and so we may choose $F$ so that $\orP_{x-1}^{(1,1)}\not\in\cF^{(1,1)}$.
Note that $\olP_x^{(1,1)}\not\in\cF^{(1,1)}$.
Also, when $(r-2)\le (n+1)/2$, we know that $|\olP_i^{(1,1)}\cap\orP_{i}^{(1,1)}|<2$ for all $i$, and so at most one of $\{\olP_{i}^{(1,1)},\orP_{i}^{(1,1)}\}$ is in $\cF^{(1,1)}$ for all $i$.
Since $|\cF^{(1,1)}|=f^{(1,1)}(r,2,1)+1=r-2$, and $\orP_{x+3}^{(1,1)}\not\in\cP^{(1,1)}$, we must have exactly one of $\{\olP_{i}^{(1,1)},\orP_{i}^{(1,1)}\}$ in $\cF^{(1,1)}$ for each $x\le i\le x+r-3$.
Inductively on $i$, this forces $\cF^{(1,1)}=\{\orP_{i}^{(1,1)}\mid x\le i\le x+r-3\}$, which is not a 2-star.
In the case that $\cF^{(1,1)}$ is a 2-star, we would have $\orP_{x+2}^{(1,1)}\not\in\cP^{(1,1)}$, and so induction would force $\cF^{(1,1)}=\{\orP_{i}^{(1,1)}\mid x\le i\le x+r-4\}$, which is a 2-star on $\{v_{x+r-4}^0,v_{x+r-3}^0\}$.

Now suppose that it is not the case that $s=2$, $t=1$, and $(a,b)=(1,1)$.
Assume that $|\cF|=f(r,s,t)$;
then $|\cF^{(a,b)}|=f^{(a,b)}(r,s,t)$ for all $(a,b)$.
Suppose that $F\in\cF^{(a,b)}$.
Then $F=\orP_x^{(j,k)}$ for some $x$, $j$, and $k$, with $(a,b)=(\mathbbm{1}_j,\mathbbm{1}_k)$.
Because $r<\left\lfloor (n+s-1)/2\right\rfloor$, we have $\cF^{(a,b)}\not=\cP^{(a,b)}$, and so we may choose $F$ so that $\orP_{x-1}^{(j',k')}\not\in\cF^{(a,b)}$ for any $j',k'$.
Note that $\olP_x^{(j',k')}\not\in\cF^{(a,b)}$ for every $j',k'$.
Also, when $(r-a-b)\le (n+s-1)/2$, we know that $|\olP_i^{(j,k)}\cap\orP_{i}^{(j',k')}|<s$ for all $i, j, k, j',$ and $k'$, and so at most $t(a,b)$ of $\{\olP_{i}^{(j,k)},\orP_{i}^{(j',k')}\}_{j,k,j',k'}$ is in $\cF^{(a,b)}$ for all $i$.
Since $|\cF^{(a,b)}|=f^{(a,b)}(r,s,t)$,  and $\orP_{x+s}^{(j,k)}\not\in\cP^{(a,b)}$ for all $j$ and $k$, we must have exactly $t(a,b)$ of $\{\olP_{i}^{(j,k)},\orP_{i}^{(j',k')}\}_{j,k,j',k'}$ in $\cF^{(a,b)}$ for each $x\le i\le x+r-s-a-b$.
Inductively on $i$, this forces $\cF^{(a,b)}=\{\orP_{i}^{(j,k)}\mid x\le i\le x+r-s-a-b, I(j,k)=(a,b)\}$, which is an $s$-star on $\{v_{x+r-s-a-b}^0,\ldots,v_{x+r-1-a-b}^0\}$.



To show that $\cF$ is a star, all $\cF^{(a,b)}$ must share a common star center.
Suppose that $\cF^{(a,b)}$ is a star at $X$, and $\cF^{(c,d)}$ is a star at $Y\not=X$.
By symmetry and relabeling of indices, we may assume that $X=\{v_0^0,\ldots,v_{s-1}\}$ and
$Y=\{v_y^0,\ldots,v_{y+s-1}^0\}$, for some $y\le n/2$.
Define $z$ to be the maximum index $i$ (allowing index $0$ to be renamed $n$ for this purpose) such that $\cF^{(a,b)}\supseteq \orP_i^{(j,k)}\not\ni v_{y+s-1}^0$ for some $j$ and $k$.
Because $r\le n/2$ we have $\orP_y^{(j',k')}\cap \orP_i^{(j,k)}\subseteq Y-\{v_{y+s-1}\}$ for some $j'$ and $k'$, and so $|\orP_y^{(j',k')}\cap \orP_i^{(j,k)}|<s$, a contradiction.
Hence $Y=X$.

Now suppose that $\cF^{(a,b)}$ is a star at $X$, and $\cF^{(0,0)}$ is not star (and therefore $s=1$ and $r=n/2)$.
Again, we may assume that $X=\{v_0^0\}$.
Thus there must be some $i$ such that $v_0^0\not\in\orP_i^{(0,0)}\in\cF^{(0,0)}$.
Now let $\orP_h^{(j,k)}\subseteq \olP_i^{(0,0)}$ for some $\orP_h^{(j,k)}\in\cF^{(j,k)}$ with $(j,k)\not=(0,0)$.
Then $\orP_0^{(j,k)}\cap\orP_i^{(0,0)}=\mt$, a contradiction.
Hence $\cF^{(0,0)}$ is a star.
\end{proof}

Additionally, a stronger result holds for bounded-length paths.

\begin{cor}
\label{c:PathDown}
Suppose that $S_n^t\not\subseteq S_3^1$.
Then $S_n^t$ is strictly $\cP^{\le k}$-\ekr\ for all $1\le k\le n/2$ and $t\ge 0$.
\end{cor}

\begin{proof}
For $k\le\floor{n/2}$, let $\cP^{\le k}$ be the family of all paths of length at most $k-1$ in $S_n^t$.
Let $\cF^{\le k}$ be a maximum intersecting family of $\cP^{\le k}$.
Then let $\cP^r$ and $\cF^r$ be the corresponding subfamilies of $r$-paths.
From Fact \ref{f:edge} and Theorem \ref{t:Sun-S-IntersectingEKRunif}, we know that $S_n^t$ is strictly $\cP^{r}$-\ekr\ for all $r<\lfloor n/2\rfloor$.
By symmetry, we may choose $v=v_0^0$ for the center of each maximum star $\cP^r_v$ of $\cF^r$ in that range.
Then we have
\[|\cF^{\le k}|
\le \sum_{1\le r\le k}|\cF^r| 
\le \sum_{1\le r\le k}|\cP^r_v|
= |\cP^{\le k}_v|.\]
If $|\cF^{\le k}|=|\cP^r_v|$ then, since $v\in \cF^1$, $v\in F$ for all $F\in\cF^{\le k}$.
\end{proof}

The following theorem shows that the \ekr\ property for the collection of all of its paths (rather than just those of uniform length) does not hold for suns other than the cycle.

\begin{theorem}
    \label{t:SunEKRnon}
    The sun $\snt$ is $\cP$-\ekr\ if and only if $t=0$, and not strictly so when $t=0$.
\end{theorem}

\begin{proof}
Let $C$ be the cycle in $\snt$, and for any path $P$ on $C$ denote by $\Bar{P}$ the {\it path complement} of $P$ on $C$; i.e., the path induced by $V(C)-V(P)$.
Let $\cF$ be an intersecting family of maximum size.

Suppose that $t=0$.
Then $|\cP|=n^2$.
Given that $\cF$ is intersecting, for every $P\in\cP$ at most one of $P$ or $\Bar{P}$ is in $\cF$.
Note that if $P$ has length $n-1$ then $\Bar{P}$ does not exist.
Let $\cS_0$ be the star of paths that contain $v_0^0$.
Then $\cS_0$ consists of exactly one path from each pair $\{P,\Bar{P}\}$, and so $|\cS_0|=\frac{n^2-n}{2}+n = \frac{n^2+n}{2}$.
Thus when $t=0$, the intersecting family of paths is $\ekr$, and the maximum size is $\frac{n^2+n}{2}$.

Interestingly, the star is not the only intersecting family of maximum size.
One can define a Hilton-Milner-type family $\cH$ as follows.
Set $P_0$ to be the path on the singleton $v_0^0$, with $\Bar{P_0}$ defined as above.
Then $\cH := \cS_0 - \{P_0\} \cup \{\Bar{P_0}\}$ is an intersecting non-star with $|\cH|=|\cS|$.
Hence $\cS_n^0=C_n$ is \ekr\ but not strictly so.
In fact, for each pair $\{P,\Bar{P}\}$ of complementary paths of $C_n$, we can build an intersecting family of maximum size by choosing one of the two paths of size at least $n/2$.

Now suppose that $t>0$.
We proceed to show that $\snt$ is not \ekr.

For a path $Q\in\cP(\snt)$ we define its {\it image} $\im(Q)$ to be the graph intersection $Q\cap C$ (consisting of those vertices and edges in both $Q$ and $C$), and denote $\cP(P)=\{Q\in\cP\mid\im(Q)=P\}$, with $\cF(P)=\cP(P)\cap\cF$.
Also define the {\it image} of $\cF$ as $\im(\cF)=\{\im(Q)\mid Q\in\cF\}$.
Clearly $|\cP(P)| = (t+1)^2$ for all $P\in C$ having at least one edge.
If the path $P$ is a single vertex, however, we have $|\cP(P)| = \binom{t+1}{2}+1$.
Thus we obtain
\begin{equation}
    \label{e:LiftSize}
    |\cP(P)| \le (t+1)^2
\end{equation}
for all $P\in C$, with equality if and only if $P$ is not a singleton.
Of course, if $\cF$ is intersecting of maximum size, then we have that $|\cF(P)|=|\cP(P)|$ for all $P\in \cF$.
In particular,
\begin{equation}
    \label{e:Fsize}
    |\cF| = \sum_{P\in\im(\cF)} |\cP(P)|.
\end{equation}
Because $\cF$ is intersecting, so is $\cF':=\im(\cF)$.
From case $t=0$ above, we know that $|\cF'|\le \frac{n^2+n}{2}$, and so
\begin{equation}
    \label{e:UpBd}
    |\cF|\le \left(\frac{n^2+n}{2}\right)(t+1)^2.
\end{equation}

In the case that $\cF$ is a star, we have that $\cF'$ is also a star, say $\cS$, and so (\ref{e:LiftSize}) and (\ref{e:Fsize}) imply that 
\begin{align*}
    |\cF|
    &= \sum_{P\in\cS}|\cP(P)|\\
    &= \binom{t+1}{2} + 1 + \left(\frac{n^2+n}{2}-1\right)(t+1)^2\\
    &< \left(\frac{n^2+n}{2}\right)(t+1)^2.
\end{align*}

On the other hand, we can define $\cH^* = \cup_{P\in\cH}\cP(P)$.
Since $\cH$ is not a star, neither is $\cH^*$.
In addition, by (\ref{e:Fsize}) we have
\[
|\cH^*| = \sum_{P\in\cH}|\cP(P)| = \left(\frac{n^2+n}{2}\right)(t+1)^2.
\]

Hence $\snt$ is not $\cP$-\ekr\ when $t>0$.
\end{proof}

Observe that, while sharing a similar notion with Conjecture \ref{c:Chvatal}, the $t>0$ case of Theorem \ref{t:SunEKRnon} does not refute it.
This is because not every vertex-subset of a path is a path; that is, $\cP$ is not subset-closed.

We finish this section with the following Hilton-Milner \cite{HiltMiln} type result, determining both the size and structure of the maximum size non-star intersecting families on the cycle in the relevant range for $r$.

\begin{theorem}
\label{t:HilMilCycle}
Let $(n+3)/3\le r\leq n/2$ and suppose that $\cF\subseteq\cP^r(C_n)$ is an intersecting non-star family of maximum-size.
Then there exists $S=\{v_a,v_b,v_c\}\subset V(C_n)$ such that $\cF = \{F\in\cP^r(C_n)\ \mid\ |F\cap S|=2\}$; moreover $|\cF| = 3r-n$.
\end{theorem}

\begin{proof}
Let $\cF\subseteq\cP^r(C_n)$ be an intersecting non-star family.
By Theorem \ref{t:trees}, we know that $\cup\cF$ is not a tree; i.e., $\cup\cF=C_n$.
Suppose that $\cF'$ is the smallest subfamily of $\cF$ such that $\cup\cF'=C_n$.
Choose any $F\in\cF'$, write $\cF''=\cF'-\{F\}$, and note that $\cup\cF''$ is a path.
Let $F_1,\ldots,F_k$ be the sequence of sets of $\cF''$ in clockwise order.
By minimality, $k=2$; hence we may write $\cF'=\{A,B,C\}$, and reiterate that $A\cap B\cap C=\mt$ and $A\cup B\cup C=V(C_n)$.
(Recall that these are intersections of vertices.)
Fixing such $A$, $B$, and $C$, we first define $X=A\cap B$, $Y=A\cap C$, and $Z=B\cap C$, and then define
\begin{itemize}
    \item 
    $\cA = \{D\in\cF\ \mid\ D\cap Z = \mt\}$ and $A' = \cap\cA$,
    \item 
    $\cB = \{D\in\cF\ \mid\ D\cap Y = \mt\}$ and $B' = \cap\cB$, and
    \item 
    $\cC = \{D\in\cF\ \mid\ D\cap X = \mt\}$ and $C' = \cap\cC$.
\end{itemize}
Note that $\cA$, $\cB$, and $\cC$ partition $\cF$ since no set of $\cF$ intersects all three of $A$, $B$, and $C$ (since $r\le n/2)$ and no set of $\cF$ misses two of $A$, $B$, and $C$ (e.g. if $D\cap X=\mt$ and $D\cap Y=\mt$ then $D\cap Z\not=\mt$).

For any set $D=\{v_d, v_{d+1}, \ldots, v_{d+r-1}\}\in \cP^r(C_n)$, where subscript arithmetic is performed modulo $n$, define for any $i$ the set $D+i = \{v_{d+i}, v_{d+1+i}, \ldots, v_{d+r-1+i}\}$.
Then, by maximality of $\cF$, we can choose some $A_0$, $B_0$, and $C_0$ so that $\cA = \{A_0+i\}_{i=0}^j$, $\cB = \{B_0+i\}_{i=0}^k$, and $\cC = \{C_0+i\}_{i=0}^l$, for some $j$, $k$, and $l$.  We shall also use $A_i=A_0+i$, $B_i=B_0+i$, and $C_i=C_0+i$, for $0\le i\le k$.

If $A'\cap B'=\emptyset$ then $A_0\cap B_k\subseteq C$; i.e., $A_0\cap B_k\cap C \neq \mt$.
But $A_0\cap Z =\mt$, and so there is some $f\le j$ and some $g\le k$ such that $|A_f\cap B_g\cap C|=1$.  This implies that $n\le 2r-1\le n-1$, a contradiction. Hence $A'\cap B'\not=\mt$; similarly, $A'\cap C'\not=\mt$ and $B'\cap C'\not=\mt$.

If $|A'\cap B'|>1$ then consider $B_{k+1}$.
We have $B_{k+1}\cap C\not=\mt$ since $B_{k+1}\cap C'\not=\mt$.
But also $B_{k+1}\cap A'\not=\mt$, which implies the contradiction that $\cF$ is not maximum.

Hence $|A'\cap B'|=1$; similarly, $|A'\cap C'|=1$ and $|B'\cap C'|=1$.
Therefore, we have $A'\cap B'=\{v_c\}$, $A'\cap C'=\{v_b\}$, and $B'\cap C'=\{v_a\}$ for some $S=\{a,b,c\}\subseteq V(C_n)$, and so $\cF=\{D\ \mid\ |D\cap S|=2\}$.

Finally, since $\cF$ is maximum, $\cA=\{D\in\cP^r\ \mid \{b,c\}\subseteq D\}$, $\cB=\{D\in\cP^r\ \mid \{a,c\}\subseteq D\}$, and $\cC=\{D\in\cP^r\ \mid \{a,b\}\subseteq D\}$.
Now let $x=(b-a)\mod{n}$, $y=(c-b)\mod n$, and $z=a-c\mod n$.
Then $|\cF| = |\cA| + |\cB| + |\cC| = (r-z) + (r-y) + (r-x) = 3r-n$.
\end{proof}


\subsection{Theta Graphs}
\label{ss:Theta}

Let $k\ge 2$ and $1\le a_1\le \cdots \le a_k$ be given, with $a_2\ge 2$.
We specify two vertices $u,v$, and for each $i$ let $P_{(i)}$ be the $u,v$-path $P_{a_i}=w_{i,0}w_{i,1}\cdots w_{i,a_i-1},$ $w_{i,a_i}$, with $u=w_{i,0}$ and $v=w_{i,a_i}$.
Then define the {\it generalized theta graph} $\Theta(a_1,\ldots,a_k)=\cup_{i=1}^k P_{(i)}$.
Each path $P_{(i)}$ is called a {\it strand} of $\Theta$, while $u$ and $v$ are called its {\it hubs} --- see Figure \ref{fig:theta_graph}.
Furthermore define the function
\[
f_k(r) = 
\begin{cases}
k+\binom{k}{2}(r-2)
&\text{ when } 3\le r\le a_1+1, \text{ and}\\
(r-a_1-2)(k-1)^2 + (a_1+2)(k-1) + (r-2)\binom{k-1}{2}
&\text{ when } a_1+2\le r\le a_2-1.\\
\end{cases}
\]

When $k=2$, the graph $\Theta(a_1,a_2)=C_{a_1+a_2}$ is a cycle, and the result is equivalent to Katona's Cycle Lemma \cite{Katona}, so we may assume that $k\ge 3$.
The case $r=2$ is already handled by Fact \ref{f:edge}, so we may assume that $r\ge 3$.

\begin{theorem}
\label{t:ThetaEKR}
Let $k\ge 3$ and $3\le r\leq (a_1+a_2+1)/2$.
Then $\Theta$ is strictly $\cP^r$-\ekr.
In particular, if $\cF$ is an intersecting subfamily of $\cP^r(\Theta)$ then $|\cF|\le f_k(r)$, with equality if and only if $\cF$ is a star centered on one of the hubs of $\Theta$.
\end{theorem}

\begin{proof}
Note that the condition on $r$ ensures that the intersection of two paths in $\cF$ is a path (i.e., their union does not contain a cycle).
The result is also true for $r<(a_1+a_2+3)/3$ by Corollary \ref{c:TreeGirthEKR}.
Set $\Theta=\Theta(a_1,\ldots,a_k)$.

First we calculate the sizes of stars in $\Theta$ to find which is the largest; i.e., write $\cP=\cP(\Theta)$ and find $|\cP_x^r|$ for all $x\in V(\Theta)$.
We begin by counting $|\cP_u^r|$, which equals
$|\cP_v^r|$ by symmetry.
\begin{itemize}
\item 
If $3\le r\le a_1+1$ then $|\cP^r_u|=k+\binom{k}{2}(r-2)$.
Indeed, there are $k$ paths in $\cP^r_u$ with endpoint $u$.
For any other $P\in\cP^r_u$, $u$ is one of the $r-2$ interior points of $P$, and the two subpaths of $P$ on opposite sides of $u$ lie in two different strands $P_{(i)}$ and $P_{(j)}$ for some $\{i,j\}\in\binom{[k]}{2}$.
\item 
If $a_1+1<r\le (a_1+a_2+1)/2$ then $a_1<r<a_2$.
For a path $P\in\cP$ with $v\in P$ we have $P_{(1)}\subset P$ and there are two cases.
If one endpoint of $P$ is $u$ or $v$ then the other endpoint is in one of the $k-1$ other strands, and so there are $2(k-1)$ such paths.
Otherwise, $u$ is one of $r-a_1-2$ vertices on the interior of $P$ and the two endpoints of $P$ are on any of the $k-1$ other strands independently, and so there are $(r-a_1-2)(k-1)^2$ such paths.
If $v\not\in P$ then there are two cases as well.
If $u$ is one endpoint of $P$ then the other endpoint is in one of the other strands, and so there are $k-1$ such paths.
Otherwise, $u$ is an interior vertex of $P$.
Here, we have two subcases.
If one endpoint of $P$ is in $P_{(1)}$ then it is one of the $a_1-1$ interior vertices of $P_{(1)}$ and the other endpoint is on any of the other $k-1$ strands, and so there are $(a_1-1)(k-1)$ such paths.
Otherwise, the two endpoints $x_1$ and $x_2$ are on any pair of paths $P_{(i)}$ and $P_{(j)}$ (respectively), with $1<i<j$, and $u$ is any of the $r-2$ interior vertices of $P$, and so there are $(r-2)\binom{k-1}{2}$ such paths.
Hence $|\cP_u^r| = 2(k-1) + (r-a_1-2)(k-1)^2 + (k-1) + (a_1-1)(k-1) + (r-2)\binom{k-1}{2}$.
\end{itemize}

Next we calculate $|\cP_{w_{i,j}}^r|$, with $j\not\in\{0,a_i\}$, and show that in each case it is strictly smaller than $|\cP_u^r|$.
By symmetry we may assume that $j\le (a_i+1)/2$.
\begin{itemize}
\item 
If $3\le r\le j+1$ then $|\cP_{w_{i,j}}^r|=r$ because $w_{i,j}$ could be any of the $r$ vertices of $P$, and $P\subseteq P_{(i)}$.
Because $k,r\ge 3$, we have $|\cP_u^r|= k+\binom{k}{2}(r-2) \ge 3+3(r-2)>r=|\cP_{w_{i,j}}^r|$, which handles the $r\le a_1+1$ case.
Similarly, when $r>a_1+1$, we have $|\cP_u^r|= (r-a_1-2)(k-1)^2 + (a_1+2)(k-1) + (r-2)\binom{k-1}{2} \ge 4(r-a_1-2)+2(a_1+2)+(r-2) = 5r-2a_1-6 > r$.
\item 
If $j+1<r\le a_i-j+1$ then $|\cP_{w_{i,j}}^r|=(r-j-1)(k-1)+(j+1)$ because $r-j-1$ of the paths $P$ extend beyond $u$ to included interior vertices of some other strand $P_{(h)}$, while the remaining paths are all contained in the strand $P_{(i)}$.
Note that $(r-j-1)(k-1)+(j+1) 
= r(k-1)-(j+1)(k-2)
< r(k-1) 
\le k + \binom{k}{2}(r-2)$ because $k\ge 3$.
Hence, in this case, $|\cP_{w_{i,j}}^r| < |\cP_u^r|$ when $r\le a_1+1$.
Additionally, $r\ge 3$ implies that $r(k-1) < 2(k-1)+\binom{k-1}{2}(r-2)$, and so $|\cP_{w_{i,j}}^r| < |\cP_u^r|$ when $r>a_1+1$ as well.
\item 
If $a_i-j+1<r\le a_1+1$ then $|\cP_{w_{i,j}}^r|=(2r-a_i-2)(k-1)+(a_i+2-r)$ because $a_i+2-r$ paths are contained entirely in the strand $P_{(i)}$, all of which contain $w_{i,j}$, while all others paths containing $w_{i,j}$ extend beyond $u$ or $v$ to include interior vertices of some other strand, and there are $(2r-a_i-2)(k-1)$ of them (i.e., $r-(a_i-r-2)$ positions remaining on an $r$-path for $w_{i,j}$ to occupy, and $(k-1)$ choices of strand for each to extend to). 
Now we aim to prove that $f_k(r)>|\cP_{w_{i,j}}^r|$.

Since $r,k\ge 3$, we know that 
$k(k-3)(r-2)+2(r-2) > 0$, 
which implies that 
$6k+2r+k^2(r-2) > 3rk+4$.
From this, one can show that
$8k+6r+(r-1)(2k-4)+k^2(r-2) > 5rk+8$, 
and the fact that $a_i\ge r-1$ yields
$8k+6r+(r-1)(2k-4)+k^2(r-2) > 5rk+8$.
Finally, this can be rewritten as
$2k+k(k-1)(r-2) > 2(2r-a_i-2)(k-1)+2(a_i+2-r)$, 
which is equivalent to 
$f_k(r)>|\cP_{w_{i,j}}^r|$.

\item 
If $a_1+2\le r\le (a_1+a_2+1)/2$, then  $|\cP_{w_{i,j}}^r|=(2r-a_i-2)(k-1)+(a_i+2-r)$ when $i\ge 2$, as above.
For $i=1$, we have $|\cP_{w_{1,j}}^r|=(r-a_1-2)(k-1)^2+(a_1+2)(k-1)$ because $r-a_1-2$ paths extend beyond both $u$ and $v$ to included interior vertices of two other strands $P_{(h)}$ and $P_{(h')}$ (possibly $h=h'$), while the remaining $a_1+2$ paths extend beyond either $u$ or $v$, exclusively.
Clearly, $|\cP_{w_{i,j}}^r|\le |\cP_{w_{1,j}}^r|$ and, since $f_k(r)$ doesn't depend on $i$, we only need to show that $|\cP_{w_{1,j}}^r|<f_k(r)$.

Indeed, $|\cP_{w_{1,j}}^r| = (r-a_1-2)(k-1)^2+(a_1+2)(k-1)<(r-a_1-2)(k-1)^2 + (a_1+2)(k-1) + \binom{k-1}{2}(r-2) = |\cP^r_u|$ for such $r$ because $k\ge 3$.
\end{itemize}

Second, let $\cF\subset \cP^r$ be a maximum intersecting family, and suppose that $\cF$ is not a star.
We proceed to prove that $|\cF|< f_k(r)$.
\begin{enumerate}
\item 
Suppose that no path in $\cF$ contains both $u$ and $v$.
Thus $\cF_u\cap\cF_v=\emptyset$.
\begin{enumerate}
\item 
Suppose that some path $P\in\cF$ lies entirely in the strand $P_{(i)}-\{u,v\}$ for some $i$.
Because $\cF$ is not a star there must be paths $P',P''\in\cF$ such that $P\cap P'\cap P''=\emptyset$, and so $P\cup P'\cup P''=P_{(i)}\cup P_{(j)}$, for some $j$.
We may suppose that $P'\in\cF_u$ and $P''\in\cF_v$ and that $P'$ and $P''$ are chosen so that $|P'\cap P''|$ is minimum among such pairs of paths with these properties.
From this is follows that $|P'\cap P''|=1$.
Indeed, otherwise define $Q$ by shifting $P'$ by one around the cycle $P_{(i)}\cup P_{(j)}$, toward $P$ and away from $P''$.
Now the pair $Q$ and $P''$ have the above properties with $|Q\cap P''| < |P'\cap P''|$, a contradiction.

Let $P'\cap P''=w_{j,x}$, for some $x$, and let the other endpoints of $P'$ and $P''$ be $w_{i,x'}$ and $w_{i,x''}$, respectively.
Define $\cF_{(i)}$ to be those paths of $\cF$ that are contained in the strand $P_{(i)}$.
(Note that $|\cF_{(i)}|\ge 1$ in this case.)
Then define $\cF_u^j$ to be the set of paths in $\cF_u-\cF_{(i)}$ that are contained in the cycle $P_{(i)}\cup P_{(j)}$, and let $\cF_u^-=\cF_u-\cF_u^j$; define $\cF_v^j$ and $\cF_v^-$ analogously.
Thus $\cF$ is partitioned into $\cF_u\cup\cF_v\cup\cF_{(i)}$ and, more finely, into $\cF_{(i)}\cup\cF_u^j\cup\cF_u^-\cup\cF_v^j\cup\cF_v^-$.

Of course, since every path in $\cF_{(i)}$ is contained in the strand $P_{(i)}$, we have 
\begin{equation}
\label{e:EasyBd}
1\le |\cF_{(i)}| \le a_i - r + 1,
\end{equation}
which implies that $a_i\ge r\ge 3$.
Moreover, every path in $\cF_{(i)}$ must contain both $w_{i,x'}$ and $w_{i,x''}$ as well, and so $|\cF_{(i)}| \le r - (x''-x')$.
Now notice that 
\begin{align*}
a_i+a_j 
&= |P_{(i)}\cup P_{(j)}|\\
&= |P'\cup P''| + (x''-1-x')\\
&= (2r-1) + (x''-x) - 1,
\end{align*}
and so 
\begin{equation}
\label{e:xprimes}
x''-x'= (a_i+a_j) - (2r-2),
\end{equation}
which yields 
\begin{equation}
\label{e:AltBd}
|\cF_{(i)}|\le 3r - (a_i+a_j) -2.
\end{equation}

Next, suppose that $|\cF_u^j|>1$ and let $Q\in\cF_u^j$ be such that, among paths in $\cF_u^j$, $|Q\cap P''|$ is maximum.
Let $w_{i,y}$ be the endpoint of $Q$ on the strand $P_{(i)}$.
Now consider the family $\cS$ of $(k-2)$ $r$-paths with endpoint $w_{i,y+1}$, containing $v$, and having their other endpoint not in the cycle $P_{(i)}\cup P_{(j)}$ (thus the $k-2$ choices of strand for the endpoint).
Because each path of $\cS$ is disjoint from $Q$, $\cS\cap\cF=\emptyset$.
Then the family $(\cF-\{Q\})\cup\cS$ is intersecting, not a star, and at least as large as $\cF$ (because $k\ge 3$).
Thus we may assume that $|\cF_u^j|=1$.
By the symmetric argument we may also assume that $|\cF_v^j|=1$.
At this point, we have that
\begin{align}
|\cF|
&= |\cF_{(i)}| + |\cF_u^j| + |\cF_u^-| + |\cF_v^j| + |\cF_v^-|\nonumber\label{e:Falmost}\\
&\le |\cF_{(i)}| + 2 + |\cF_u^-| + |\cF_v^-|,
\end{align}
and so it remains to estimate $|\cF_u^-|$ and $|\cF_v^-|$.

Each path in $\cF_u^-$ must contain both $u$ and $w_{i,x''}$, so that it intersects both $P$ and $P''$.
Moreover, if $x''>r-2$ then $\cF_u^-=\emptyset$, while if $x''\le r-2$ then, for each $x''\le h\le r-2$, there are exactly $k-2$ such paths having endpoint $w_{i,h}$.
Because $|\cF|$ is maximum, all of these would be in $\cF_u^-$.
Therefore $|\cF_u^-|= (r-1-x'')(k-2)$.
Similarly, if $x'<a_i-r$ then $\cF_v^-=\emptyset$, while if $x'\ge a_i-r$ then $|\cF_v^-|= (r+x'-a_i-1)(k-2)$.
Hence, in all cases we have 
\begin{align}
|\cF_u^-| + |\cF_v^-|
&= [(r-1-x'') + (r+x'-a_i-1)](k-2)\nonumber\\
&= [2r-2-a_i-(x''-x')](k-2)\nonumber\\
&= [2r-2-a_i-(a_i+a_j - 2r + 2)](k-2)\nonumber\\
&= (4r - 2a_i - a_j - 4)(k-2).
\label{e:Fminus}
\end{align}
Therefore, bounds \eqref{e:Falmost} and \eqref{e:Fminus} yield
\begin{equation}
|\cF| 
\le |\cF_{(i)}| + 2 + (4r - 2a_i - a_j - 4)(k-2).
\label{e:UpperBd}
\end{equation}

To show that the upper bound in \eqref{e:UpperBd} is at most $f_k(r)$, we consider the two cases for $r$ that define $f$.
First take $3\le r\le a_1+1$, which implies that $2r \le 2a_1+2 = (3/2)a_1 + (a_1/2+1) +1 < (3/2)a_i + a_j + 1$ since $a_1\le a_i$, $a_1\le a_j$, and $a_j\ge 3$.
From this, we obtain 
\begin{align*}
0 
< -2(r-2) + \frac{3}{2}a_i + a_j - 3
& = \left(\frac{2}{2}-\frac{7}{2}+\frac{1}{2}\right)(r-2) + \left(2-\frac{1}{2}\right)a_i + a_j - 3\\
& \le \left(\frac{k-1}{2}-\frac{7}{2}+\frac{1}{k-1}\right)(r-2) + \left(2-\frac{1}{k-1}\right)a_i + a_j - 3\\
& = \left(\frac{k}{2}-4+\frac{1}{k-1}\right)r + \left(2-\frac{1}{k-1}\right)a_i + a_j + 5 - \frac{2}{k-1} - k,
\end{align*}
because $k\ge 3$.
We continue, deriving
\[
\frac{a_i-r+3}{k-1} 
< \left(\frac{k}{2}-4\right)r + 2a_i + a_j + 4 + \frac{k}{k-1} - k,
\]
and so
\[
\frac{a_i-r+3}{k-1} + 4r - 2a_i - a_j - 4
< \frac{k}{k-1} + \frac{k}{2}(r-2).
\]
Finally, using the bound from \eqref{e:EasyBd}, we find that
\begin{align*}
|\cF|
&\le a_i-r+3 + (4r - 2a_i - a_j - 4)(k-2)\\
&\le a_i-r+3 + (4r - 2a_i - a_j - 4)(k-1)
< k + \binom{k}{2}(r-2),
\end{align*}
achieving the strict bound of $f_k(r)$ for this range of $r$.

Second, take $a_1+2\le r\le a_2-1$.
In this case we will show that $|\cF| \le f_k(r)$ by splitting each side into two parts and proving that 
\begin{enumerate}
    \item 
    $|\cF_{(i)}| + 2 \le \binom{k-1}{2}(r-2)+3$ and
    \item
    $(4r - 2a_i - a_j - 4)(k-2) < [(a_1+2) + (r-a_1-2)(k-1)](k-1)-3$.
\end{enumerate}
For part (i) recall that $r\le (a_1+a_2+1)/2$, and $k\ge 3$. We use \eqref{e:AltBd} to write $3r - (a_i+a_j) \le 3r - (a_1+a_2) \le 3r - (2r-1) = (r-2)+3 \le \binom{k-1}{2}(r-2)+3$.

For part (ii) we recall that $a_1+2\le r\le a_2-1$, which implies that $3\le a_2-a_1\le a_h-a_1$ for all $h$, so that $4<9\le 2(a_i-a_1)+(a_j-a_1)$; i.e., $3a_1+10 < 2a_i+a_j+6$.
Then
\begin{align*}
a_1(k-2)
&= (a_1+2)(k-5) - 2(k-5) + 3a_1\\
&\le r(k-5) - 2k + 3a_1 + 10\\
&< r(k-5) -2k + 2a_i + a_j + 6,
\end{align*}
so that $4r-2a_i-a_j-4 < (r-a_1-2)(k-1)+a_1$.
Therefore 
\begin{align*}
(4r - 2a_i - a_j - 4)(k-2)/(k-1) 
&< 4r - 2a_i - a_j - 4\\
&< (r-a_1-2)(k-1) + a_1\\
&< (r-a_1-2)(k-1) + (a_1+2) -3/2\\
&\le (r-a_1-2)(k-1) + (a_1+2) -3/(k-1),
\end{align*}
and hence $|\cF|<f_k(r)$ for this range of $r$ as well.

\item 
Now we know that no path in $\cF$ lies entirely in a single strand's interior $P_{(i)}-\{u,v\}$.
Thus every path in $\cF$ intersects $\{u,v\}$, and so $|\cF|=|\cF_u|+|\cF_v|$.
(Recall that $\cF_u\cap\cF_v=\emptyset$.)
Because $\cF$ is not a star, some path $P$ contains $u$ but not $v$, and some path $P'$ contains $v$ but not $u$; among such pairs, we choose one such that $|P\cap P'|$ is minimum.
Note that $|P\cap P'|=1$; otherwise, we can add another path ending at this intersection contradicting maximality.
Observe that this property requires $r\ge a_1/2$.

Having chosen $P\in\cF_u$ and $P'\in\cF_v$, let $\{i,j\}\in\binom{[k]}{2}$ be such that $P\subseteq P_{(i)}\cup P_{(j)}$.
Notice that (because $r\le (a_1+a_2+1)/2$) $\cF_v$ partitions into $\cF_v^i\cup\cF_v^j$, where $Q\in\cF_v^\ell$ if $Q\cap P\subset P_{(\ell)}$. 
Hence $|\cF_v|=|\cF_v^i|+|\cF_v^j|$.

Let $P''\in\cF_v^j$ such that $|P\cap P''|$ is minimum.
Then $|P\cap P''|=1$; indeed, suppose not.
Because $\cF$ is maximum there must be some path $P^+\in\cF_u$ such that $P^+\in P_{(i)}\cup P_{(j)}$ and $|P'\cap P^+|=2$ since, otherwise, such a path could be added to the family.
For $l\not=j$, define $P'_l$ to be the path in the cycle $P_{(j)}\cup P_{(l)}$ beginning at $w_{j,h_j}$.
Now the family $\cF - \{P^+\}\cup \{P'_1, \ldots, P'_k\} - \{P'_j\}$ is intersecting and larger than $\cF$ since $k\ge 3$, a contradiction.
Hence $|P\cap P''|=1$.

Now this implies that $|\cF_u|=1$ since any other path containing $u$ is disjoint from $P'$ or $P''$.
In the same way that we built $P^+$ by shifting $P$ toward $v$ along the strand $P_{(i)}$, we can shift $P'$ and $P''$ toward $u$ along the strands $P_{(i)}$ and $P_{(j)}$ respectively. 
Among these there are at least two, say $Q$ and $Q'$, that do not contain $u$.  
Then $\cF-\{P\}\cup\{Q,Q'\}$ is still intersecting and larger than $\cF$, a contradiction.
\end{enumerate}
\item 
Now suppose that some path $P\in\cF$ contains both $u$ and $v$.
Then there is a unique $i$ such that the strand $P_{(i)}\subset P$.
Of course, this implies that $r\ge a_i+1$, which we know is at least $a_1+1$.
Similar to the above, we partition $\cF = \cF'_u\cup \cF'_v\cup \cF'_{uv}$, where $\cF'_{uv}$ consists of all the paths of $\cF$ that contain both $u$ and $v$, $\cF'_{u}$ contains those paths containing $u$ but not $v$, and $\cF'_{v}$ contains those paths containing $v$ but not $u$.
We will enumerate $|\cF| = |\cF'_u| + |\cF'_v| + |\cF'_{uv}|$ in two cases as follows.

If $\cup\cF$ is a tree, then by Theorem \ref{t:trees} every intersecting family is a star. Hence we may assume that there are paths  $P'\in\cF'_u$ and $P''\in\cF'_v$ which intersect on some other strand $P_(j)$ with $j\neq i$; we choose such paths with $|P'\cap P''|$ minimized. As we have seen previously, by maximality it must be that $|P'\cap P''|=1$, and so we write $P'\cap P''=\{w_{j,x}\}$.
\begin{enumerate}
    \item 
    First, let $r = a_1+1$, so that $|\cF'_{uv}|=1$.
    Thus we may assume that $i=1$.
    (Note that $(a_1+a_j+1)/2\ge r= a_1+1$, and so $2a_1-a_j\le a_1-1$, which we will use below.)
    Now we have $|\cF'_u| = (k-1)(r-x-1) + 1$ and $|\cF'_v| = (k-1)(r-a_j+x-1) + 1$.
    
    Hence 
    \begin{align*}
        |\cF|\ 
        &=\ |\cF'_u| + |\cF'_v| + |\cF'_{uv}|\ 
        =\ (k-1)(2r-a_j-2) + 3\\
        &=\ (k-1)(2a_1-a_j) + 3\ 
        \le\ (k-1)(a_1-1) + k\\
        &<\ k + \binom{k}{2}(a_1-1)\ 
        =\ k + \binom{k}{2}(r-2)\ 
        =\ f_k(r).
    \end{align*}
    \item 
    Second, let $a_1+2\le r\le (a_1+a_2+1)/2$.
    Then $|\cF'_{uv}| = 2(k-1) + (r-a_i-2)(k-1)^2$ because $k-1$ paths end at each of $u$ and $v$, while all others have each endpoint on some other strand different from $P_{(i)}$.
    As in the first case, we have $|\cF'_u| = (k-1)(r-x-1) + 1$ and $|\cF'_v| = (k-1)(r-a_j+x-1) + 1$, and so
    \begin{align}
        |\cF|
        &= |\cF'_u| + |\cF'_v| + |\cF'_{uv}|\nonumber \\
        &= 2(k-1) + (r-a_i-2)(k-1)^2 + (k-1)(2r-a_j-2) + 2.
        \label{e:rhs}
    \end{align}
    Now $f_k(r) = (r-a_1-2)(k-1)^2 + (a_1+2)(k-1) + (r-2)\binom{k-1}{2}$ in this case, so the right side of \eqref{e:rhs} is less than $f_k(r)$ when
    \begin{equation}
    \label{e:newineq}
    0  < (a_i-a_1)(k-1) + \frac{1}{2}(k-2)(r-2)+(a_1+a_j-2r+2) - \frac{2}{k-1}.
    \end{equation}
    Now $a_1+2\le r\le (a_1+a_2+1)/2$ implies that $a_2\ge a_1+3$, and so $a_i-a_1\ge a_2-a_1\ge 3$.
    It also implies that $a_1+a_j-2r+2\ge a_j-a_2+1\ge 1$.
    Hence, the right hand side of Equation \eqref{e:newineq} is at least $6+(1/2)+1-1>0$, which finishes the proof of part (b), and hence of case (2).
\end{enumerate}
\end{enumerate}

Therefore, $\cF$ is a star and $|\cF|\le f_k(r)$, with equality if and only if it is centered on a hub.
\end{proof}


\section{Transversal and Triangular Results}
\label{s:Trans}

In this section, we collect some related results on transversal numbers which we believe to be of independent interest. We include some results on finite projective planes, since they play an important role in many extremal combinatorial contexts.
For example, in \cite{ErdLov}, Erd\H{o}s and Lov\'asz defined $m(r)$ to be the minimum size of an intersecting family of $r$-sets with transversal number $r$. 
They showed that $m(r)\ge \frac{8}{3}r-3$, but obtained as an upper bound $m(r)\le4r^{3/2}\log{r}$  -- their construction uses random collections of lines in projective planes, and so held when $r-1$ is the order of a projective plane. 
This was first improved by Kahn\cite{Kahn1} to $m(r)<Cr\log{r}$, again contingent on the  existence of an order $r-1$ projective plane. Finally, Kahn\cite{Kahn2} showed that there is some prime power $K$ so that $m(r)\le 5(K^2+k)t$ when $r=Kq+t$, thus confirming the Erd\H{o}s-Lov\'asz conjecture that $m(r)=O(r)$.  
It is worth noting that Kahn's method does not provide any idea how large $K$ might be.

In Subsection \ref{ss:Elementary}, we relax the condition that $\cF$ must be uniform, proving some basic transversal results involving degree conditions for our intersecting families.  
In Subsection \ref{ss:Triangular} we focus on those intersecting families that contain no large stars whatsoever, with $\Delta(\cF)=2$.  
Finally, in Subsection \ref{ss:ProjPlanes}, we construct some examples in projective planes which demonstrate the sharpness of several of our preceding results in Section \ref{s:Trans}.


\subsection{Elementary Observations}
\label{ss:Elementary}

First, we note some straightforward bounds that we will use later. 
These results are well known, but we include proofs here for completeness.

\begin{fact}[See \cite{Jukna}, p. 111, Sec. 10.2.]
\label{f:MinSetUpperBound}
If $\cF$ is an intersecting family of sets then $\t(\cF)\le\min_{A\in\cF}|A|$.
\end{fact}

\begin{proof}
Let $F$ be a set in $\cF$ of minimum size.
Since $\cF$ is intersecting, if follows that $F$ is a transversal of $\cF$.
Hence $\t(\cF)\le|F|$.
\end{proof}

\begin{fact}
\label{f:basicUpperBound}
If $\cF$ is an intersecting family of sets then $\t(\cF)\le\ceil{|\cF|/2}$.
\end{fact}

\begin{proof}
For any intersecting family of sets $\cF$, pair off its sets arbitrarily, with possibly one set remaining. 
For each pair of sets choose one element from their intersection.
Also choose one element from the remaining set if it exists.
These choices form a transversal of size $\ceil{|\cF|/2}$. 
Therefore $\t(\cF)\le\ceil{|\cF|/2}$.
\end{proof}

\begin{fact}[See \cite{Jukna}, p. 111, Sec. 10.2.]
\label{f:basicLowerBound}
If $\cF$ is an intersecting family of sets then $\t(\cF)\ge\ceil{|\cF|/\D(\cF)}$.
\end{fact}

This follows from a more general statement in \cite{Jukna}, stating that $\t(\cF)$ is bounded below by the size of the largest matching in $\cF$. 
We give a direct proof here.

\begin{proof}
Let $X$ be a transversal of $\cF$.
For any $x\in X$ we get $\deg(x)\le\D(\cF)$.
Hence $X$ must contain at least $\ceil{|\cF|/\D(\cF)}$ elements to be a transversal of $\cF$.
\end{proof}


\subsection{Triangular Families}
\label{ss:Triangular}

In the interest of exploring \ekr\ problems, we will be interested in trios of sets with no common intersection.  In particular, we say that a trio of sets $\{A,B,C\}$ is {\it triangular} if $A\cap B\cap C=\mt$.
We further call a family $\cF$ of sets {\it triangular} if $\{A,B,C\}$ is triangular for all $\{A,B,C\}\subseteq\cF$.
For example, any family of lines in general position in the plane is both intersecting and triangular.
For subsets of a finite set, one may ask how large an intersecting triangular family can be or what is its transversal number.
We begin with a characterization.

\begin{fact}
\label{f:TriangularDelta}
An intersecting family $\cF$ containing at least two sets is triangular if and only if $\D(\cF)=2$.
\end{fact}

\begin{proof}
Assume that $\cF$ is triangular.
That is, for all $\{A,B,C\}\subseteq \cF$ we get that $A\cap B\cap C=\emptyset$, which implies $\D(\cF)<3$.
Furthermore, since $\cF$ contains at least two sets and is intersecting $\D(\cF)>1$.
Hence $\D(\cF)=2$.

Suppose that $\D(\cF)=2$.
Then for any $x\in\cup\cF$, we get that $x$ is an element of at most 2 sets in $\cF$.
Therefore any three sets from $\cF$ are triangular.
Hence $\cF$ is triangular.
\end{proof}

\begin{fact}
\label{f:MaxSetSizeBound}
If $\cF$ is a triangular intersecting family of sets then $|\cF|\le 1+ \min_{A\in\cF}|A|$.
\end{fact}

\begin{proof}
Let $A$ be a set in $\cF$ of minimum size.
Since $\cF$ is triangular, it follows from Fact \ref{f:TriangularDelta} that $\D(\cF)=2$.
For all $x\in A$, $x$ is contained in at most one other set in $\cF$.
Hence $|\cF|\le 1+|A|$.
\end{proof}

The following corollary shows that Fact \ref{f:basicUpperBound} is best-possible.

\begin{cor}
\label{c:NoTriples}
If $\cF$ is a triangular intersecting family of sets then $\t(\cF)=\ceil{|\cF|/2}$.
\end{cor}

\begin{proof}
Since $\cF$ is intersecting, it follows from Fact \ref{f:basicUpperBound} that $\t(\cF)\le\ceil{|\cF|/2}$.
As for the lower bound, from Fact \ref{f:basicLowerBound} we get that $\t(\cF)\ge\ceil{|\cF|/\D(\cF)}$.
Given that $\cF$ is triangular, Fact \ref{f:TriangularDelta} gives us $\D(\cF)=2$.
Thus $\t(\cF)=\ceil{|\cF|/2}$.
\end{proof}

It is worth asking whether or not triangular families characterize equality in Fact \ref{f:basicUpperBound}.
Curiously, this is true for odd $|\cF|$ but not necessarily for all even $|\cF|$.

\begin{theorem}
\label{t:OddFam}
If $\cF$ is an intersecting family of $2k-1$ sets with $\t(\cF)=k$ then $\cF$ is triangular.
\end{theorem}

\begin{proof}
Suppose that $\cF$ is intersecting, $|\cF|=m=2k-1$, and $\t(\cF)=k$.
If $m\le 2$ then $\cF$ is triangular by definition, so we assume that $m\ge 3$, and let $X=\{x_1,\ldots,x_k\}$ be a minimum transversal.

If some $a\in\cup\cF$ has $\deg(a)>2$ then, for $\cF'=\cF-\cF_a$, we have $|\cF'|\le m-3=2k-4$, and so $\t(\cF')\le k-2$ by Fact \ref{f:basicUpperBound}.
Now let $Y$ be a minimum transversal of $\cF'$; then $Y\cup\{a\}$ is a transversal of $\cF$ of size at most $k-1$, a contradiction.
Hence $\D(\cF)=2$.
\end{proof}

For $m\in\{4,6\}$, there are examples of intersecting families $\cF$ with $|\cF|=m$, $\t(\cF)=m/2$, and $\D(\cF)>2$.
Indeed, for a given integer $h>1$, subset $S\subset\{0,\ldots,h-1\}=[h]$, and positive integer $i$, define the set $S^i=\{s+i\mod h\mid s\in S\}$.
Also define the family $\cR_h(S)=\{S^i\mid i\in [h]\}$.
Now set $S_2=\{0,1,2\}$ and $S_3=\{0,1,3\}$.
It is easy to verify that, for $k\in\{2,3\}$, the family $\cR_{2k}(S_k)$ is intersecting and has size $2k$.
Additionally, $\t(\cR_{2k}(S_k))=k$ because, for any $T\in\binom{[2k]}{k-1}$, there is some $S^i\in\cF_{2k}$ with $S^i\cap T=\mt$.
Moreover, $\cR_{2k}(S_k)$ is not triangular because it is 3-regular.

This simple rotational construction illustrates the conflicting objectives in the Erd\H{o}s-Lov\'asz problem discussed at beginning of this section.
When $|S|$ is large compared to $2k$, $\cR_{2k}(S)$ tends to be intersecting with small transversal number, while if $|S|$ is small, $\cR_{2k}(S)$ tends to have large transversal number but not be intersecting.


\subsection{Projective Planes}
\label{ss:ProjPlanes}

The finite projective plane $\zpq$, where $q$ is any prime power, is a rich example in extremal combinatorics. 
Frankl and Graham \cite{FranGrah} gave a sufficient condition for a uniform family to be \ekr\ that considers its generalized degrees.
They used a construction based on projective planes to show that the condition couldn't be weakened.

Recall the notation $\D_s(\cF) = \max_{|X|=s}|\cF_X|$.

\begin{theorem}[\cite{FranGrah}]
\label{t:FranGrah}
If $\cF\subseteq \binom{[n]}{r}$ is intersecting and  $\D_1(\cF)\ge (r^2-r+1)\D_2(\cF)$ then $\cF$ is \ekr.
When $r-1$ is a prime power, there exists a non-\ekr\  $\cH\subseteq \binom{[n]}{r}$ with $\D_1(\cH)= r^2-r$ and $\D_2(\cH)=1$.
\hfill $\Box$
\end{theorem}

The construction of $\cH$ in Theorem \ref{t:FranGrah} has $n=m(r-1)(r^2-r+1)$ for some large enough $m$, and is built from a collection of mutually orthogonal copies of $\zprm$.

Here we use projective planes to investigate the tightness of Fact \ref{f:MaxSetSizeBound}.
Observe that Fact \ref{f:MaxSetSizeBound} implies that every triangular subfamily of $\zpq$ has size at most $q+2$.
In fact, this is tight for even $q$ but can be improved to $q+1$ for odd $q$, as we show below.

We first introduce some notation for $\zpq$, as follows.
Let $\zF_q$ denote the finite field of order $q$, and set $V=V_q=((\zF_q\cup\{\w\})\times\zF_q)\cup\{(\w,\w)\}$.
The set of $q^2+q+1$ elements of $V$ serve as both the points of $\zpq$ and the indices of the lines of $\zpq$; that is, $\zpq=\{L_\a\ \mid\ \a\in V\}$, where
\begin{itemize}
    \item 
    $L_{(m,b)}= \{(x,y)\in\zF_q\times\zF_q\ \mid y=mx+b\}\cup\{(\w,m)\}$, for $(m,b)\in\zF_q\times\zF_q$,
    \item
    $L_{(\w,b)}= \{(b,y)\in\zF_q\times\zF_q\}\cup\{(\w,\w)\}$, for $b\in\zF_q$, and
    \item
    $L_{(\w,\w)}= \{(\w,y)\ \mid y\in\zF_q\}\cup\{(\w,\w)\}$,
\end{itemize}
and all arithmetic, unless specified otherwise, is performed in $\zF_q$ throughout the remainder of this section.
Of course, $\deg(\a)=|L_\a|=q+1$ for all $\a\in V$.
It is well known that $\zpq$ is exactly 1-intersecting.

\begin{lemma}
\label{l:ProjPlaneCover}
For any prime power $q$, $\tau(\zpq)=q+1$.
\end{lemma}
\begin{proof}
The upper bound follows from Fact \ref{f:MinSetUpperBound}.
The lower bound follows from Fact \ref{f:basicLowerBound}.
\end{proof}

\begin{theorem}
\label{t:ProjPlane}
Let $q$ be an odd prime power.
If $\cF$ is a triangular subfamily of $\zpq$ of maximum size then $|\cF|=q+1$.
\end{theorem}

\begin{proof}
We first construct an example to show that triangular families of size $q+1$ exist. 
For each $b\in\zF_q-\{0,1\}$, define $m_b=\frac{-b}{1-b}$.
Then, we define our family of $q+1$ lines to be $$\cF~=~\{L_{(\w,\w)},L_{(\w,0)},L_{(0,0)},L_{(m_2,2)}\ldots,L_{(m_{q-1},q-1)}\}.$$

Note that every line in $\cF$ has a distinct slope; hence we just need to show that every point $(x,y)$ has $\deg((x,y))\leq 2$.
Let $(x,y)\in L_{(m_b,b)}$. 
Then $y=x m_b+b$, and so
\[y=-x\frac{b}{1-b}+b,\]
\[(1-b)(y-b)=-x b,\text{ and }\]
\[b^2+(x-y-1)b+y=0.\]
The final equation is quadratic in $b$ and thus has at most two solutions in $\zF_q$, and so $\deg((x,y))\leq2$. Therefore $\cF$ is a triangular family of size $q+1$, showing the lower bound.

For the upper bound, consider a triangular subfamily $\cF\subseteq\zP(q)$ with $|\cF|>q+1$.
By Fact \ref{f:MaxSetSizeBound}, $|\cF|=|L|+1=q+2$.
By symmetry, we may assume that $\cF$ contains the lines $L_{(\w,\w)},L_{(\w,0)},L_{(0,0)}$, so we define $\cF'=\cF-\{L_{(\w,\w)},L_{(\w,0)},L_{(0,0)}\}$.

Because $\cF$ is 1-intersecting and triangular, and $|\cF'|=q-1$, we know that for every $b\in\zF_q^*=\zF_q-\{0\}$ there exists a unique $m_b\in\zF_q^*$ such that $L_{(m_b,b)}\in\cF'$.
Moreover, there exists a unique $a_b\in\zF_q^*$ such that $(a_b,0)\in L_{(m_b,b)}$.
Hence $m_b a_b+b=0$; i.e.,  $b=-m_b a_b$.

Let $\gamma$ be a generator for $\zF_q^*\cong \zZ_{q-1}$, so that every $z\in\zF_q^*$ satisfies $z=\gamma^e$ for some $e\in\zZ_{q-1}=\{0,\ldots,q-2\}$.
Write $\zF_q^*=\{b_1,\ldots,b_{q-1}\}$, and relabel $m_{b_i}$ and $a_{b_i}$ as $m_i$ and $a_i$, respectively.
For each $1\le i\le q-1$ define $d_i$, $e_i$, and $f_i$ uniquely by $b_i=\gamma^{d_i}$, $m_i=\gamma^{e_i}$, and $a_i=\gamma^{f_i}$, so that $d_i=e_i+f_i\mod{(q-1)}$.
Hence $\{b_i\}=\{m_i\}=\{a_i\}=\zF_q^*$, so that $\{d_i\}=\{e_i\}=\{f_i\}=\zZ_{q-1}$.
That is, there exist permutations $(d_1\cdots d_{q-1})$, $(e_1\cdots e_{q-1})$, and $(f_1\cdots f_{q-1})$ of $\zZ_{q-1}$ such that $d_i=e_i+f_i\mod q-1$ for all $i$.
But, since $q$ is odd, this yields the contradiction that
$$\frac{q-1}{2}=\sum_{i=1}^{q-1} d_i=\sum_{i=1}^{q-1} e_i+\sum_{j=1}^{q-1} f_j=\frac{q-1}{2}+\frac{q-1}{2}=0\mod q-1$$
Hence $|\cF'|< q-1$, and so $|\cF|\leq q+1$.
\end{proof}

\begin{theorem}
\label{t:Power2ProjPlane}
Let $q=2^t$ for $t\in\mathbb{N}$.
If $\cF$ is a triangular subfamily of $\zpq$ of maximum size then $|\cF|=q+2$.
\end{theorem} 

\begin{proof}
Given that all lines in $\zpq$ contain $q+1$ points, from Fact \ref{f:MaxSetSizeBound} we get $|\cF|\leq q+2$.
Now we need only show the lower bound by displaying a triangular family of size $q+2$.

Using the lower bound $\cF'$ constructed in Theorem \ref{t:ProjPlane}, let $\cF=\cF'\cup\{L_{(1,1)}\}$.
Here we must show for any $(x,y)\in L_{(1,1)}$ there is exactly one $L_{(m_b,b)}\in\cF$ containing $(x,y)$.
That is, over $\zF_2$ we have
\begin{align*}
-x+1&=-x\frac{b}{1-b}+b,\\
0&=b^2+(2x-2)b-x+1
=b^2+x+1, \text{ and }\\
b&=\sqrt{x+1},
\end{align*}
which is unique in characteristic 2.
Thus there is exactly one line $L_{(m_b,b)}\in\cF$ containing $(x,y)$.
Therefore $\cF$ is a triangular family of size $q+2$, showing the lower bound, and thus $|\cF|=q+2$.
\end{proof}


\section{More Questions}
\label{s:Quests}

Define the $d^{\it th}$ {\it power} of a graph $G$, written $G^{(d)}$, to have $V(G^{(d)})=V(G)$ and an edge $xy$ for every pair of vertices with $\dist_G(x,y)\le d$.

\begin{prob}
\label{p:OtherG}
Study the $\cP^r$-\ekr\ properties of other graphs such as $C_n^{(d)}$ or subdivisions of complete graphs or complete bipartite graphs.
(Note that $C_n$ is a subdivision of $K_{2,2}$ when $n\ge 4$ and the theta graph $\Theta(a_1,\ldots,a_k)$ is a subdivision of $K_{2,k}$ when each $a_i\ge 2$.)
Also consider graphs $C_n^f$ for $f:V(C_n)\rightarrow\mathbb{N}$, or the more general class of trees plus an edge; i.e., {\it unicyclic} graphs.
\end{prob}

We note that Frankl, et al. \cite{FHIKLMP} addresses the \ekr\ problem for spanning trees of $K_n$.
Along these lines, it would be equally interesting to study the case of spanning paths of $K_n$ (i.e., $r=n$ in Problem \ref{p:OtherG}).

A family $\cF$ is called {\it Sperner} if there are no two sets $A,B\in\cF$ with $A\subseteq B$. Bollobas' inequality \cite{Bineq} gives a very strong condition on the sizes of Sperner families. There is a somewhat complicated history here; Bollob\'as' inequality is a general version of a sharper-in-some-cases result of Lubell, which had been earlier discovered post-Bollob\'as but pre-Lubell by Meshalkin, and which can in hindsight be read out of an earlier result of Yamamoto; in order to better reflect all four author's various contributions in chronological order, this is perhaps best called the YBLM inequality. Proving YBLM-type theorem in this $\cH$-EKR context would be a significant step forward; we provide a starting question here.

\begin{qst}
Assume $\cF\subseteq\cP(C_n)$ is a maximum size family that is both intersecting and Sperner. Is it true that $\cF$ must be uniform?
\end{qst}

It may be both interesting and necessary to consider cross-intersecting families of paths in pursuit of answering the questions above; bounding the sum or product of the sizes of such cross-intersecting families is also useful and relevant in its own right.

Theorem \ref{t:Sun-S-IntersectingEKRunif} states that, among the largest intersecting families of $\cP^r(C_n)$ having transversal number at least 2 is one having transversal number exactly 2.

\begin{qst}
Given a graph $G$, is it true that among the largest intersecting families in $\cP^r(G)$ (or $\cP(G)$) with transversal number at least $k$ is one with transversal number exactly $k$?
\end{qst}

Finally, we ask the question left open by Theorem \ref{t:OddFam}. There, we show that for odd sized intersecting families of $2k-1$ sets, that the families achieving $\tau(\cF)=k$ are the triangular families. We also showed that this is not necessarily the case for all even $n$, giving examples when $|\cF|\in\{4,6\}$.

\begin{qst}
\label{q:EvenFam}
For which $k$ does there exist an intersecting family $\cF$ of size $2k$ with $\tau(\cF)=k$ and $\Delta(\cF)>2$?  How can we characterize such families?
\end{qst}


\begin{thebibliography}{10}
\label{bibliography}

\bibitem{AhlsKhac}
R. Ahlswede and L. Khachatrian,
{\it The complete intersection theorem for systems of finite sets},
European J. Combin. \textbf{18} (1997), no. 2, 125--136.

\bibitem{BergeGraphs}
C. Berge, \emph{Graphs and Hypergraphs}. Vol. 6. North-Holland Publishing Company (1973).

\bibitem{Berge}
C. Berge, 
{\it Nombres de coloration de l’hypergraphe h-partie complet},
in: Hypergraph Seminar, Columbus, Ohio 1972, Springer, New York, 1974, 13--20.

\bibitem{Bineq}
B. Bollob\'as,
{\it On generalized graphs},
Acta Math. Hungar. \textbf{16} (1965), no. 3--4, 447-452.

\bibitem{BollLead}
B. Bollobás, I. Leader, 
{\it An Erd\H{o}s-Ko-Rado theorem for signed sets},
Comput. Math. Appl. \textbf{34} (1997), no. 11, 9--13.

\bibitem{BorgFegh}
P. Borg and C. Feghali,
{\it The Hilton-Spencer cycle theorems via Katona's shadow intersection theorem},
Discuss. Math. Graph Theory \textbf{43} (2023), no. 1, 277--286.

\bibitem{Chvat}
V. Chv\'atal,
{Intersecting families of edges in hypergraphs having the hereditary property},
Hypergraph seminar, Lecture Notes in Math. 411 (Springer- Verlag, Berlin, 1974) 61--66.

\bibitem{DezaFran}
M. Deza, P. Frankl, 
{\it Erd\H{o}s-Ko-Rado Theorem --- 22 years later}, 
SIAM J. Algebr. Discrete Methods \textbf{4} (1983), no. 4, 419--431.

\bibitem{EllFilFri}
D. Ellis, Y. Filmus, and Ehud Friedgut, {\it Triangle-intersecting families of
graphs}, J. Eur. Math. Soc. \textbf{14 (3)} (2012), 841–-885.

\bibitem{ErdKoRad}
P. Erd\H{o}s, C. Ko, and R. Rado,
{\it Intersection theorems for systems of finite sets},
Quart. J. Math. Oxford Ser. (2) \textbf{12} (1961), 313--320.

\bibitem{ErdLov}
P. Erd\H{o}s, L. Lov\'asz, {\it Problems and results on 3-chromatic hypergraphs and some related questions}, Colloq. Math. Soc. János Bolyai \textbf{10} (1974), 609--627.

\bibitem{FegHurKam}
C. Feghali, G. Hurlbert, and V. Kamat,
{\it An Erd\H{o}s-Ko-Rado theorem for unions of length 2 paths},
Discrete Math. \textbf{343} (2020), no. 12, 112121, 6 pp.

\bibitem{FranGrah}
P. Frankl and R.L. Graham,
{\it Old and new proofs of the Erd\H{o}s-Ko-Rado Theorem},
Sichuan Daxue Xuebao \textbf{26} (1989), 112--122.

\bibitem{FranHurl}
P. Frankl and G. Hurlbert,
{\it On the Holroyd-Talbot conjecture for sparse graphs}, Discrete Math. {\bf{347}}-1, (2024).

\bibitem{FHIKLMP}
P. Frankl, G. Hurlbert, F. Ihringer, A. Kupavskii, N. Lindzey, K. Meagher, V. Pantangi, {\it Intersecting Families of Spanning Trees}, arxiv:2502.08128.

\bibitem{HiltMiln}
A. J. W. Hilton and E. C. Milner, 
{\it Some intersection theorems for systems of finite sets},
Quart. J. Math. Oxford \textbf{18}, no. 2, (1967) 369--384.

\bibitem{HolSpeTal}
F.C. Holroyd, C. Spencer, J. Talbot, 
{\it Compression and Erd\H{o}s–Ko–Rado graphs}, 
Discrete Math. \textbf{293} (2005), no. 1--3, 155--164.

\bibitem{HolrTalb}
F.C. Holroyd, J. Talbot, 
{\it Graphs with the Erd\H{o}s–Ko–Rado property},
Discrete Math. \textbf{293} (2005), no. 1--3, 165--176.

\bibitem{Hurlbert}
G. Hurlbert, {\it A Survey of the Holroyd-Talbot Conjecture}, arxiv:2501.16144
(2025).

\bibitem{Jukna}
S. Jukna,
Extremal combinatorics. With applications in computer science ($2^{\it nd}$ ed.).
{\it Texts in Theoretical Computer Science}.
Springer, Heidelberg (2011).

\bibitem{Kahn1}
J. Kahn, {\it On a Problem of Erd\H{o}s and Lov\'asz: random lines in a projective plane}, Combinatorica \textbf{12} (1992), 417--423.

\bibitem{Kahn2}
J. Kahn, {\it On a Problem of Erd\H{o}s and Lov\'asz II: n(r)=O(r)}, J. Amer. Math. Soc. \textbf{7} (1) (1994), 125--143.

\bibitem{Katona}
G.O.H. Katona,
{\it A simple proof of the Erd\H{o}s-Ko-Rado theorem},
J. Combin. Theory Ser. B \textbf{13} (1972), no. 2, 183--184.

\bibitem{LLLYM}
D. Lubell,
{\it A short proof of Sperner's lemma},
J. Combin. Theory \textbf{1} (1966), no. 2, 299.

\bibitem{MLYM}
L.D. Meshalkin,
{\it Generalizations of Sperner's theorem on the number of subsets of a finite set (in Russian)}, Teor. Probab. Ver. Primen. \textbf{8} (1963), 219-220.  English translation in Theory of Prob. and its Appl. \textbf{8} (1964), 204--205.

\bibitem{SimoSos1} M. Simonovits, V. S\'os, {\it  Intersection theorems for graphs}, in Problemes combinatoires et th\'eorie des graphes, volume 260, Colloq. Internat. CNRS, Paris, (1978), 389–391. .

\bibitem{SimoSos2} M. Simonovits, V. S\'os, {\it  Intersection theorems for graphs. II.}, in
Combinatorics (Proc. Fifth Hungarian Colloq., Keszthely, 1976), Vol. II, Colloq. Math. Soc. J\'anos Bolyai, 18 North-Holland Publishing Co.,
Amsterdam-New York, (1978), 1017–1030.

\bibitem{YLYM}
K. Yamamoto,
{\it Logarithmic order of free distributive lattices},
J. Math. Soc. Japan \textbf{6} (1954), 343--353.

\end{thebibliography}
\end{document}